\documentclass[reqno]{amsart}

\usepackage{amssymb,amsmath}
\usepackage{color}

\definecolor{axcolor}{rgb}{.3,0,.3}

\newtheorem{corollary}{Corollary}
\newtheorem{theorem}{Theorem}
\newtheorem{lemma}{Lemma}
\newtheorem{proposition}{Proposition}
\newtheorem{question}{Question}

\newcommand{\limit}{\mathsf{Limit}} \newcommand{\diff}{\mathsf{Diff}}

\newcommand{\vx}{\mathbf{\bar x}}
\newcommand{\vy}{\mathbf{\bar y}}

\newcommand{\vz}{\mathbf{\bar z}}
\newcommand{\de}{:=}
\newcommand{\rac}{\mathbb{Q}}
\newcommand{\R}{\mathbb{R}}
\newcommand{\A}{\mathbb{A}}

\newcommand{\defiff}{\ \stackrel{\;def}{\Longleftrightarrow}\ }

\newcommand{\IOb}{\ensuremath{\mathsf{IOb}}} 
\newcommand{\Ob}{\ensuremath{\mathsf{Ob}}} 
\newcommand{\B}{\ensuremath{\mathit{B}}} 
\newcommand{\Ph}{\ensuremath{\mathsf{Ph}}} 
\newcommand{\Q}{\ensuremath{\mathit{Q}}} 
\newcommand{\W}{\ensuremath{\mathsf{W}}} 
\newcommand{\ev}{\ensuremath{\mathsf{ev}}} 
\newcommand{\dom}{\ensuremath{\mathsf{Dom}\,}} 
\newcommand{\ran}{\ensuremath{\mathsf{Ran}\,}} 

\newcommand{\lc}{\ensuremath{\mathsf{lc}}}
\newcommand{\w}{\ensuremath{\mathsf{w}}}
\newcommand{\num}{\textit{Num}}

\newcommand{\ax}[1]{\textcolor{axcolor}{\ensuremath{\mathsf{#1}}}} 

\begin{document}
\title[numbers needed for modeling accelerated observers]{What properties of numbers are needed to model accelerated observers in relativity?}

\author{Gergely~Sz\'ekely} 

\address{Alfr\'ed R\'enyi Institute of Mathematics \\ Hungarian
  Academy of Sciences \\ Re\'altanoda utca 13-15, H-1053, Budapest,
  Hungary}

\email{szekely.gergely@renyi.mta.hu}

\date{\today}

\begin{abstract}
We investigate the possible structures of numbers (as physical
quantities) over which accelerated observers can be modeled in special
relativity. We present a general axiomatic theory of accelerated
observers which has a model over every real closed field. We also show
that, if we would like to model certain accelerated observers, then
not every real closed field is suitable, e.g., uniformly accelerated
observers cannot be modeled over the field of real algebraic
numbers. Consequently, the class of fields over which uniform
acceleration can be investigated is not axiomatizable in the language
of ordered fields.
\end{abstract}

\maketitle
\section{Introduction}

In this paper within an axiomatic framework, we investigate the
possible structures of numbers (as physical quantities) over which
accelerated observers can be modeled in special relativity.

There are several reasons for this kind of investigations. One of them
is that we cannot experimentally verify whether the structure of
quantities is isomorphic to $\R$ (the field of real numbers). Thus we
cannot have any direct empirical support to leave out of consideration
the several other algebraic structures.  Another reason is that these
investigations lead to a deeper understanding of the relation of our
mathematical and physical assumptions.  For a more general perspective
of this research direction, see \cite{wnst}.

In general we would like to investigate the question
\begin{center}
{\it ``What structure can numbers have in a certain physical theory?''}
\end{center}
To introduce our central concept, let \ax{Th} be a theory of physics.
In this case, we can introduce notation
$\num(\ax{Th})$ for the class of the possible quantity structures of
theory \ax{Th}:
\begin{multline*}
\num(\ax{Th}):=\{\mathfrak{Q}:\mathfrak{Q} \text{ is a structure of
  quantities} \text{over which \ax{Th} has a model.}\}
\end{multline*}

In this paper, our main question of interest is that what algebraic
properties have to be satisfied by the numbers as physical quantities
if we want to model accelerated observers in special relativity.  So
we will restrict our investigation to the case when \ax{Th} is a
theory of special relativity extended with accelerated
observers. However, this question can be investigated in any other
physical theory the same way.

We introduce several theories and axioms of relativity theory.  For
example, our axiom system for $d$-dimensional special relativity
(\ax{SpecRel_d}, see p.\pageref{specrel}) captures the
kinematics of special relativity perfectly (if $d\ge3$) as it implies that the
worldview transformations between inertial observers are Poincar\'e
transformations, see \cite{wnst}.  Without any extra assumptions
\ax{SpecRel_d} has a model over every ordered field, i.e.,
\begin{equation*}
\num(\ax{SpecRel_d})=\{\mathfrak{Q}:\mathfrak{Q} \text{ is an ordered
  field} \}.
\end{equation*}
Therefore, \ax{SpecRel_d} has a model over the field of rational numbers
$\rac$, too.  However, if we also assume that inertial observes can
move with arbitrary speed less than that of light, then every positive
number has to have a square root if $d\ge 3$, see
\cite{wnst}. In particular, the number structure cannot be the
field of rational numbers, but it can be the field of real algebraic
numbers.

If we assume only that inertial observers can move only approximately
with any speed slower than that of light, then we still can model
special relativity over $\rac$, see \cite{MSz-wnst}. 

Moving toward general relativity we will see that our theory of
accelerated observes (\ax{AccRel_d}, see p.\pageref{accrel}) requires
the structure of quantities to be a real closed field, i.e., an
ordered field in which every positive number has a square root and
every odd degree polynomial has a root, see
Theorem~\ref{thm-rc}. Specially, \ax{AccRel_d} does not have a model
over $\rac$. However, any real closed field, e.g., the field of real
algebraic numbers, can be the quantity structure of \ax{AccRel_d}.

If we extend \ax{AccRel_d} by an extra axiom stating that there are
uniformly accelerated observers (\ax{Ax\exists UnifOb}, see
p.\pageref{axunifob}), then the field of real algebraic numbers cannot
be the structure of quantities anymore if $d\ge3$, see
Theorem~\ref{thm-noalg}.  A surprising consequence of this result is
that $\num(\ax{AccRel_d}+\ax{Ax\exists UnifOb})$ is not a first-order
logic axiomatizable class of fields, see Corollary~\ref{cor-nondef}. That
is, in the language of ordered fields, it is impossible to axiomatize
those fields over which uniformly accelerated observers can be
modeled.

An interesting and related approach of Stannett introduces two
structures one for the measurable numbers and one for the theoretical
numbers and assumes that the set of measurable numbers is dense in the
set of theoretical numbers, see \cite{Mike}.

We chose first-order predicate logic to formulate our axioms because
experience (e.g., in geometry and set theory) shows that this logic is
the best logic for providing an axiomatic foundation for a theory. A
further reason for choosing first-order logic is that it is a well
defined fragment of natural language with an unambiguous syntax and
semantics, which do not depend on set theory. For further reasons,
see, e.g., \cite[\S Why FOL?]{pezsgo}, \cite{ax}, \cite[\S 11]{SzPhd},
\cite{vaananen}, \cite{wolenski}.

\section{The language of our theories}
\label{lang-s}

To our investigation, we need an axiomatic theory of spacetimes.\footnote{In this paper, we
will use the language and axiom systems of \cite{Synthese}.} The first
important decision in writing up an axiom system is to choose the set
of basic symbols of our logic language, i.e., what objects and
relations between them we will use as basic concepts.

Here we will use the following two-sorted\footnote{That our theory is
  two-sorted means only that there are two types of basic objects
  (bodies and quantities) as opposed to, e.g., Zermelo--Fraenkel set
  theory where there is only one type of basic objects (sets).}
language of first-order logic (FOL) parametrized by a natural number
$d\ge 2$ representing the dimension of spacetime:
\begin{equation}
\{\B,\Q\,; \Ob, \IOb, \Ph,+,\cdot,\le,\W\},
\end{equation}
where $\B$ (bodies) and $\Q$ (quantities) are the two sorts,
$\Ob$ (observers), $\IOb$ (inertial observers) and $\Ph$ (light
signals) are one-place relation symbols of
sort $\B$, $+$ and $\cdot$ are two-place function symbols of sort
$\Q$, $\le$ is a two-place relation symbol of sort $\Q$, and $\W$ (the
worldview relation) is a $d+2$-place relation symbol the first two
arguments of which are of sort $\B$ and the rest are of sort $\Q$.

Relations $\Ob(o)$, $\IOb(m)$ and $\Ph(p)$ are translated as
``\textit{$o$ is an observer},'' ``\textit{$m$ is an inertial
  observer},'' and ``\textit{$p$ is a light signal},''
respectively. To speak about coordinatization of observers, we
translate relation $\W(k,b,x_1,x_2,\ldots,x_d)$ as ``\textit{body $k$
  coordinatizes body $b$ at space-time location $\langle x_1,
  x_2,\ldots,x_d\rangle$},'' (i.e., at space location $\langle
x_2,\ldots,x_d\rangle$ and instant $x_1$).

{\bf Quantity  terms} are the variables of sort $\Q$ and what can be
built from them by using the two-place operations $+$ and $\cdot$,
{\bf body terms} are only the variables of sort $\B$.
$\IOb(m)$, $\Ph(p)$, $\W(m,b,x_1,\ldots,x_d)$, $x=y$, and $x\le y$
where $m$, $p$, $b$, $x$, $y$, $x_1$, \ldots, $x_d$ are arbitrary
terms of the respective sorts are so-called {\bf atomic formulas} of
our first-order logic language. The {\bf formulas} are built up from
these atomic formulas by using the logical connectives \textit{not}
($\lnot$), \textit{and} ($\land$), \textit{or} ($\lor$),
\textit{implies} ($\rightarrow$), \textit{if-and-only-if}
($\leftrightarrow$) and the quantifiers \textit{exists} ($\exists$)
and \textit{for all} ($\forall$).

To make them easier to read, we omit the outermost universal
quantifiers from the formalizations of our axioms, i.e., all the free
variables are universally quantified.

We use the notation $\Q^n$ for the set of all $n$-tuples of elements
of $\Q$. If $\vx\in \Q^n$, we assume that $\vx=\langle
x_1,\ldots,x_n\rangle$, i.e., $x_i$ denotes the
$i$-th component of the $n$-tuple $\vx$. Specially, we write $\W(m,b,\vx)$ in
place of $\W(m,b,x_1,\dots,x_d)$, and we write $\forall \vx$ in place
of $\forall x_1\dots\forall x_d$, etc.

We use first-order logic set theory as a meta theory to speak about model
theoretical terms, such as models, validity, etc.  The {\bf models} of
this language are of the form
\begin{equation}
{\mathfrak{M}} = \langle \B, \Q;\Ob_\mathfrak{M}, \IOb_\mathfrak{M},\Ph_\mathfrak{M},+_\mathfrak{M},\cdot_\mathfrak{M},\le_\mathfrak{M},\W_\mathfrak{M}\rangle,
\end{equation}
where $\B$ and $\Q$ are nonempty sets, $\Ob_\mathfrak{M}$,
$\IOb_\mathfrak{M}$ and $\Ph_\mathfrak{M}$ are subsets of $\B$,
$+_\mathfrak{M}$ and $\cdot_\mathfrak{M}$ are binary functions and
$\le_\mathfrak{M}$ is a binary relation on $\Q$, and $\W_\mathfrak{M}$
is a subset of $\B\times \B\times \Q^d$.  Formulas are interpreted in
$\mathfrak{M}$ in the usual way.  For the precise definition of the
syntax and semantics of first-order logic, see, e.g., \cite[\S
  1.3]{CK}, \cite[\S 2.1, \S 2.2]{End}.

\section{Numbers required by special relativity}
\label{ax-s}
 
First we formulate axioms for special relativity concerning inertial
observers only in the logic language of Section~\ref{lang-s}.

The key axiom of special relativity states that the speed of light is
the same in every direction for every inertial observers.

\begin{itemize}
\item[] \underline{\ax{AxPh}}: For any inertial observer, the speed of
  light is the same everywhere and in every direction (and it is
  finite). Furthermore, it is possible to send out a light signal in
  any direction (existing according to the coordinate system)
  everywhere:
\begin{multline}
\IOb(m)\rightarrow \exists c_m\Big[c_m>0\land \forall \vx\vy
 \Big(\exists p \big[\Ph(p)\land \W(m,p,\vx)\land
    \W(m,p,\vy)\big] \\\leftrightarrow (x_2-y_2)^2+\ldots +(x_d-y_d)^2=
  c_m^2\cdot(x_1-y_1)^2\Big)\Big].\footnotemark
\end{multline}
\end{itemize}
\footnotetext{That is, if $m$ is an inertial observer, there is a is a
  positive quantity $c_m$ such that for all coordinate points $\vx$
  and $\vy$ there is a light signal $p$ coordinatized at $\vx$ and
  $\vy$ by observer $m$ if and only if equation $ (x_2-y_2)^2+\ldots
  +(x_d-y_d)^2= c_m^2\cdot(x_1-y_1)^2$ holds.}  

To get back the intended meaning of axiom \ax{AxPh} (or even to be
able to define subtraction from addition), we have to assume some
properties of numbers. 

In our next axiom, we state some basic properties of addition,
multiplication and ordering true for real numbers.

\begin{itemize}
\item[]
\underline{\ax{AxOField}}:
 The quantity part $\langle \Q,+,\cdot,\le \rangle$ is an ordered field, i.e.,
\begin{itemize}
\item[$\bullet$]  $\langle\Q,+,\cdot\rangle$ is a field in the sense of abstract
algebra; and
\item[$\bullet$] 
the relation $\le$ is a linear ordering on $\Q$ such that  
\begin{itemize}
\item[i)] $x \le y\rightarrow x + z \le y + z$ and 
\item[ii)] $0 \le x \land 0 \le y\rightarrow 0 \le xy$
holds.
\end{itemize}
\end{itemize}
\end{itemize}

Using axiom \ax{AxOFiled} instead of assuming that the structure of
quantities is the field of real numbers not just makes our theory more
flexible, but also makes it possible to meaningfully investigate our
main question.  Another reason for using \ax{AxOField} instead of $\R$
is that we cannot experimentally verify whether the structure of
physical quantities are isomorphic to $\R$. Hence the assumption that
the structure of quantities is $\R$ cannot have a direct empirical
support.  The two properties of real numbers which are the most
difficult to defend from empirical point of view are the Archimedean
property, see \cite{Rosinger08}, \cite[\S 3.1]{Rosinger09},
\cite{Rosinger11a}, \cite{Rosinger11b}, and the supremum
property,\footnote{The supremum property (i.e., that every nonempty
  and bounded subset of the numbers has a least upper bound) implies
  the Archimedean property. So if we want to get ourselves free from
  the Archimedean property, we have to leave this one, too.} see the
remark after the introduction of \ax{CONT} on p.\pageref{p-cont}.

We also have to support \ax{AxPh} with the assumption that all
observers coordinatize the same ``external" reality (the same set of
events).  By the {\bf event} occurring for observer $m$ at point
$\vx$, we mean the set of bodies $m$ coordinatizes at $\vx$:
\begin{equation}
\ev_m(\vx)\de\{ b : \W(m,b,\vx)\}.
\end{equation}

\begin{itemize}
\item[]
\underline{\ax{AxEv}}
All inertial observers coordinatize the same set of events:
\begin{equation*}
\IOb(m)\land\IOb(k)\rightarrow  \exists \vy\, \forall b
\big[\W(m,b,\vx)\leftrightarrow\W(k,b,\vy)\big].
\end{equation*}
\end{itemize}
From now on, we will use $\ev_m(\vx)=\ev_k(\vy)$ to abbreviate the
subformula $\forall b [\W(m,b,\vx)\leftrightarrow\W(k,b,\vy)]$ of
\ax{AxEv}. 

The three axioms above are enough to capture the essence of special
relativity. However, let us assume two more simplifying axioms.

\begin{itemize}
\item[]\underline{\ax{AxSelf}}
Any inertial observer is stationary relative to himself:
\begin{equation*}
\IOb(m)\rightarrow \forall \vx\big[\W(m,m,\vx) \leftrightarrow x_2=\ldots=x_d=0\big].
\end{equation*}
\end{itemize}
Our last axiom on inertial observers is a symmetry axiom saying that
they use the same units of measurement.

\begin{itemize}
\item[]\underline{\ax{AxSymD}}
Any two inertial observers agree as to the spatial distance between
two events if these two events are simultaneous for both of them;
furthermore, the speed of light is 1 for all observers:
\begin{multline*}
\IOb(m)\land\IOb(k) \land x_1=y_1\land x'_1=y'_1\land
\ev_m(\vx)=\ev_k(\vx') \land \ev_m(\vy)=\ev_k(\vy') \rightarrow\\
(x_2-y_2)^2+\ldots +(x_d-y_d)^2=(x'_2-y'_2)^2+\ldots +(x'_d-y'_d)^2
\text{ and }\\
\IOb(m)\rightarrow\exists
p\big[\Ph(p)\land\W(m,p,0,\ldots,0)\land\W(m,p,1,1,0,\ldots,0)\big].
\end{multline*}
\end{itemize}

Let us introduce an axiom system for special relativity as the collection
of the five axioms above:
\begin{equation*}\label{specrel}
\ax{SpecRel_d} \de \ax{AxPh}+\ax{AxOField}+ \ax{AxEv}+\ax{AxSelf}+
\ax{AxSymD}.
\end{equation*}
Streamlined axiom system \ax{SpecRel_d} perfectly captures the
kinematics of special relativity since it implies that the worldview
transformations between inertial observers are Poincar\'e
transformations, see \cite{wnst}.

\section{Numbers implied by accelerated observers}

Now we are going to investigate what happens with the possible
structures of quantities if we extend our theory \ax{SpecRel_d} with
accelerated observers. To do so, let us recall our first-order logic
axiom system of accelerated observers \ax{AccRel_d}.  The key axiom of
\ax{AccRel_d} is the following:
\begin{itemize}
\item[]\underline{\ax{AxCmv}}
 At each moment of its worldline, each observer
sees the nearby world for a short while as an inertial
   observer does.  
\end{itemize}
For formalization of \ax{AxCmv} in the first-order language of
Section~\ref{lang-s}, see \cite{SzPhd}. In \ax{AccRel_d} we will also
use the following localized version of axioms \ax{AxEv} and
\ax{AxSelf} of \ax{SpecRel_d}.

\begin{itemize}
\item[]\underline{\ax{AxEv^-}} Observers coordinatize all the events in which they participate:
\begin{equation*}
\Ob(k)\land  \W(m,k,\vx)\rightarrow\exists
\vy\enskip \ev_m(\vx)=\ev_k(\vy).
\end{equation*}
\end{itemize}

\begin{itemize}
\item[]\underline{\ax{AxSelf^-}} In his own worldview, the worldline
  of any observer is an interval of the time-axis containing all the
  coordinate points of the time-axis where the observer sees
  something:
\begin{multline*}
 \big[\W(m,m,\vx)\rightarrow x_2=\ldots=x_d=0\big]
 \land\\ \big[\W(m,m,\vy)\land\W(m,m,\vz)\land x_1<t<y_1\rightarrow
   \W(m,m,t,0,\ldots,0)\big] \land\\ \exists b \big[
   \W(m,b,t,0,\ldots,0) \rightarrow \W(m,m,t,0,\ldots,0)\big].
\end{multline*}
\end{itemize}

Let us now introduce a promising theory of accelerated observers as
\ax{SpecRel_d} extended with the three axioms above.
\begin{equation*}
\ax{AccRel^0_d}\de \ax{SpecRel_d}+ \ax{AxCmv}+\ax{AxEv^-}+\ax{AxSelf^-}
\end{equation*}

Axiom \ax{AxCmv} ties the behavior of accelerated observers to the
inertial ones and \ax{SpecRel_d} captures the kinematics of special
relativity perfectly (it implies that the worldview transformations
between inertial observers are Poincar\'e transformations, see
\cite{wnst}). Therefore, it is quite natural to hope that
\ax{AccRel^0_d} is a strong enough theory of accelerated observers to
prove the most fundamental results about accelerated observers.
However, \ax{AccRel^0_d} does not imply even the most basic
predictions about accelerated observers such as the twin paradox or
that stationary observers measure the same time between two events
\cite{twp}, \cite[\S 7]{SzPhd}. Moreover, it can be proved that even
if we add the whole first-order logic theory of real numbers to
\ax{AccRel^0_d} is not enough to get a theory that implies the twin
paradox, see \cite{twp}, \cite[\S 7]{SzPhd}.

In the models of \ax{AccRel^0_d} in which \ax{TwP} is not true, there
are some definable gaps in $\Q$. Our axiom scheme \ax{CONT} excludes
these gaps.
\begin{itemize}
\label{p-cont}
\item[]\underline{\ax{CONT}} Every  parametrically definable,
  bounded and nonempty subset of $\Q$ has a supremum (i.e., least upper bound) with respect to $\le$.
\end{itemize}
\noindent In \ax{CONT} ``definable'' means ``definable in the language
of \ax{AccRel_d}, parametrically.'' For a precise formulation of \ax{CONT}  in the first-order language of
Section~\ref{lang-s},
see \cite[p.692]{twp} or \cite[\S 10.1]{SzPhd}.

Axiom scheme \ax{CONT} makes the supremum postulate of real numbers
closer to the physical/empirical level because \ax{CONT} speaks only
about ``physically meaningful'' sets of the quantities which can be defined
in the language of our (physical) theory and not ``any fancy subset.''

Our axiom scheme of continuity (\ax{CONT}) is Tarski's first-order
logic version of Hilbert's continuity axiom in his axiomatization of
geometry, see \cite[pp.161-162]{Gol}, fitted to the language of
\ax{AccRel_d}.

When $\Q$ is the ordered field of real numbers, \ax{CONT} is automatically
true.
Let us introduce our axioms system \ax{AccRel_d} as the extension of \ax{AccRel^0_d} by axiom scheme \ax{CONT}. 
\begin{equation*}\label{accrel}
\ax{AccRel_d}\de\ax{AccRel^0_d} + \ax{CONT}
\end{equation*}

It can be proved that axiom system \ax{AccRel_d} implies the twin
paradox, see \cite{twp}, \cite[\S 7.2]{SzPhd}.

An ordered field is called {\bf real closed field} if a first-order
logic sentence of the language of ordered fields is true in it exactly
when it is true in the field of real numbers, or equivalently if it is
an ordered field in which every positive number has a square root and
every polynomial of odd degree has a root in it, see, e.g.,
\cite{tarski-dmethod}.

Axiom scheme \ax{CONT} is so powerful that it implies that the
possible structures of quantities have to be real closed fields:
\begin{theorem}\label{thm-rc} For all $d\ge 2$, 
\begin{equation*}
\num(\ax{AccRel_d})=\{\mathfrak{Q}: \mathfrak{Q} \text{ is real
    closed fields}\}.
\end{equation*}
\end{theorem}
\noindent
For the proof of Theorem~\ref{thm-rc}, see \cite{wnst}.

\begin{question}
Can \ax{CONT} be replaced in \ax{AccRel_d} with some natural
assumptions such that they (together with \ax{AccRel^0_d}) imply all
(or certain) important predictions of relativity theory about
accelerated observers (e.g., the twin paradox) yet they do not
require that the structure of quantities is a real closed field.
\end{question}

\section{Numbers implied by uniformly accelerated observers}

In paper \cite{wnst}, we have seen that assuming existence of
observers can ensure the existence of numbers.  So let us investigate
another axiom of this kind, which postulates the existence of
uniformly accelerated observers. To introduce this axiom, let us
define the {\bf life-curve} $\lc_{km}$ of observer $k$ according to
observer $m$ as the worldline of $k$ according to $m$ {\it
  parametrized by the time measured by $k$},\label{life-curve}
formally:
\begin{equation}
\lc_{km}\de\{\langle t,\vx \rangle :
\exists \vy\enskip k\in \ev_k(\vy)=\ev_m(\vx)\land y_1=t\}.
\end{equation}

Now we can introduce our axiom ensuring the existence of uniformly
accelerated observers.
\begin{itemize}
\label{axunifob}
\item[]\underline{\ax{Ax\exists UnifOb}} It is possible to accelerate
  an observer uniformly:\footnote{In relativity theory, uniformly
    accelerated observers are moving along hyperbolas, see, e.g.,
    \cite[\S 3.8, pp.37-38]{dinverno}, \cite[\S 6]{MTW}, \cite[\S
      12.4, pp.267-272]{Rindler}.}
\begin{multline*}
\IOb(m)\rightarrow \exists k \Big[ \Ob(k) \land \dom \lc_{km}=\Q\land \lc_{km}(0)=\vy\land \lc_{km}(1)_1>y_1\land
  \forall \vx \\\big[ \vx\in \ran \lc_{km} \leftrightarrow
    (x_2-y_2)^2-(x_1-y_1)^2=a^2\land x_3-y_3=\ldots=x_d-y_d=0\big]\Big].
\end{multline*}
\end{itemize}

We use the notation $\mathfrak{Q}\in\num(\ax{Th})$ for algebraic
structure $\mathfrak{Q}$ the same way as the model theoretic notation
$\mathfrak{Q}\in Mod(\ax{AxField})$, e.g., $\rac\in\num(\ax{Th})$
means that $\rac$, the field of rational numbers, can be the
structure of quantities (numbers) in \ax{Th}.

Let $\A\cap \R$ denote the ordered field of real algebraic numbers.
Theorem~\ref{thm-noalg} states that the ordered field of algebraic real
numbers cannot be the structure of quantities of theory \ax{AccRel_d}
+ \ax{Ax\exists UnifOb} if $d\ge3$:
\begin{theorem}\label{thm-noalg}
Let $d\ge 3$. Then 
\begin{equation*}\A\cap \R\not\in\num(\ax{AccRel_d} +
\ax{Ax\exists UnifOb}).
\end{equation*}
\end{theorem}
The proof of Theorem~\ref{thm-noalg} is in Section~\ref{proof-s} on
p.\pageref{proof-noalg}.
    
Since the ordered fields of real numbers and real algebraic numbers
are elementarily equivalent, Theorem~\ref{thm-noalg} implies that
the quantity part of the models of \ax{AccRel_d} + \ax{Ax\exists UnifOb}
is not an elementary class:\footnote{Of course, it is a pseudo
  elementary class, i.e., it is a class of the reducts of an elementary
  class.}

\begin{corollary}\label{cor-nondef}
Let $d\ge 3$. Then $\num(\ax{AccRel_d} + \ax{Ax\exists UnifOb}) $ is
not axiomatizable in the language of ordered fields.
\end{corollary}

By Theorem~\ref{thm-noalg}, we know that not every real closed field
can be the quantity structure of \ax{AccRel_d} + \ax{Ax\exists UnifOb},
e.g., it cannot be the field of real algebraic numbers.  However, the
problem that exactly which ordered fields can be the quantity
structures of \ax{AccRel_d} + \ax{Ax\exists UnifOb} is still open:
\begin{question}\label{que-uob}
Exactly which ordered fields are the elements of class
$\num(\ax{AccRel_d} + \ax{Ax\exists UnifOb})$?
\end{question}
\noindent
We can also ask what properties of numbers do axiom \ax{Ax\exists
  UnifOb} requires without \ax{CONT}:
\begin{question}\label{que-uob0}
Exactly which ordered fields are the elements of class
$\num(\ax{AccRel^0_d} + \ax{Ax\exists UnifOb})$?
\end{question}
Theorem~\ref{thm-exp} below suggests that the answer to
Questions~\ref{que-uob} and \ref{que-uob0} may have something to do
with ordered exponential fields, see \cite[\S 4]{DW},
\cite{SK}.

\section{Proof of Theorem~\ref{thm-noalg}}
\label{proof-s}

In this section, we prove Theorem~\ref{thm-noalg}. To do so, let us
introduce some concepts.  The {\bf space component} of $\vx\in\Q^d$ is
defined as
\begin{equation}
\vx_s\de \langle x_2,\ldots, x_d\rangle.
\end{equation} The (signed)
{\bf Minkowski length} of $\vx\in \Q^d$ is
\begin{equation}
\mu(\vx)\de \left\{
\begin{array}{rll}
\sqrt{\rule{0pt}{11pt} x_1^2-|\vx_s|^2} & \text{ if}\quad x_1^2\ge|\vx_s|^2, \\
-\sqrt{\rule{0pt}{11pt}|\vx_s|^2-x_1^2} & \text{ in other cases, } 
\end{array}
\right.
\end{equation}
and the {\bf Minkowski distance} between $\vx$ and $\vy$ is
$\mu(\vx,\vy)\de\mu(\vx-\vy).$ We use the signed version of the
Minkowski length because it contains two kinds of information: (i) the
length of $\vx$, and (ii) whether it is spacelike, lightlike or
timelike.  Let $H\subseteq\Q$. We say that $H$ is an {\bf interval}
iff $z\in H$ when there are $x,y\in H$ such that $x<z<y$.  We say that
a function $\gamma:H \rightarrow \Q^d$ is a {\bf curve} if $H$ is an
interval and has at least two distinct elements.

The usual (first-order logic) formula can be used to define the
differentiability function over any ordered field
$\mathfrak{Q}$. The {\bf derivative of} function $f:\Q\rightarrow\Q^n$ is
$A\in\Q^n$ at $x_0\in \dom f$:
\begin{multline}
\diff(f,x_0,A)\defiff \forall \varepsilon>0\; \exists \delta
>0 \;\forall x \enskip 0<|x-x_0|<\delta \\\land x\in \dom f \rightarrow
|f(x)-f(x_0)-A(x-x_0)|<\varepsilon|x-x_0|.
\end{multline}
In the case when there is one and only one $A$ such that
$\diff(f,x_0,A)$ holds, we write $f'(x_0)=A$. It can be proved that
there is at most one $A$ such that $\diff(f,x_0,A)$ holds if $\dom f$
is open, see \cite[Thm.10.3.9]{SzPhd}

A curve $\gamma$ is called {\bf timelike curve} iff it is
differentiable, and $\gamma'(t)$ is timelike, i.e.,
$\mu\big(\gamma'(t)\big)>0$, for all $t\in \dom \gamma$.
We call a timelike curve $\alpha$
{\bf well-parametrized} if
$\mu\big(\alpha'(t)\big)=1$ for all $t\in \dom \alpha$.

\begin{theorem}\label{thm-wp}
Let $d\ge 3$.
Assume \ax{AccRel_d}.
Let $k$ be an observer and $m$ be an {\it inertial} observer.
Then $\lc_{km}$ is a well-parametrized timelike curve.
\end{theorem}
\noindent
For the proof of Theorem~\ref{thm-wp}, see \cite[Thm.6.1.11]{SzPhd}.

A part of real analysis can be generalized for arbitrary ordered
fields without any real difficulty, see \cite[\S 10]{SzPhd}. However,
a certain fragment of real analysis can only be generalized within
first-order logic for \textit{definable} functions and their proofs
need axiom schema \ax{CONT}. We refer to these generalizations by
marking them ``\ax{CONT}-.'' The first-order logic generalizations of
some theorems, such as Chain Rule can be proved without \ax{CONT}, so
they are naturally referred to without the ``\ax{CONT}-'' mark.

\begin{lemma}\label{lem-mon}
Assume \ax{AxOField} and \ax{CONT}.  Let $f:\Q\rightarrow\Q$ be an
injective definable continuous function. Then $f$ is also monotonous.
\end{lemma}
\noindent
\noindent
For the proof of Lemma~\ref{lem-mon}, see \cite[Thm.10.2.4]{SzPhd}.

\begin{proposition}\label{prop-tlc}
Assume \ax{CONT} and \ax{AxOField}. Let
$\gamma,\delta:\Q\rightarrow \Q^d$ be definable and differentiable
well-parametrized timelike curves such that
$\ran\gamma=\ran\delta$. Then there are $\varepsilon\in\{-1,+1\}$
and $c\in\Q$ such that $\delta(t)=\gamma(\varepsilon t+c)$ for all
$t\in\Q$.
\end{proposition}
\begin{proof}
By \cite[Lem.10.5.4]{SzPhd}, we have that there is a (definable)
differentiable function $h:\Q\rightarrow\Q$ such that $|h'|=1$ and
$\delta(t)=\gamma\big(h(t)\big)$ for all $t\in\Q$. By
\ax{CONT}-Darboux's Theorem~\cite[p.110]{SzPhd}, $h'(t)=1$ for all
$t\in\Q$ or $h'(t)=-1$ for all $t\in\Q$. By the \ax{CONT} version of
the fundamental theorem of integration~\cite[Prop.10.3.19]{SzPhd},
$h(t)=t+c$ or $h(t)=-t+c$ for some $c\in\Q$.
\end{proof}

\begin{lemma}\label{lem-ulc}
Let $d\ge3$.
Assume \ax{AccRel_d}.
Let $m_1$, $m_2$ be inertial observers and let  $k_1$ and $k_2$ be observers such that 
\begin{enumerate}
\item\label{dom} $\dom\lc_{k_1m_1}=\dom\lc_{k_2m_2}=\Q$, 
\item\label{ran} $\ran\lc_{m_1k_1}=\ran\lc_{m_2k_2}$,
\item\label{sp} $\lc_{k_1m_1}(0)=\lc_{k_2m_2}(0)$, 
\item\label{pd} $\lc_{k_1m_1}(1)_1>\lc_{k_1m_1}(0)_1$ and $\w_{k_2m_2}(1)_1>\w_{k_1m_1}(0)_1$.
\end{enumerate}
Then $\lc_{k_1m_1}=\lc_{k_2m_2}$.
\end{lemma}

\begin{proof}
We are going to prove our statement by applying
Proposition~\ref{prop-tlc} to $\lc_{k_1m_1}$ and $\lc_{k_2m_2}$ So let
$\gamma\de\lc_{k_1m_1}$ and $\delta\de\lc_{k_2m_2}$.  By
Theorem~\ref{thm-wp}, $\gamma$ and $\delta$ are well-parametrized
timelike curves.  By assumptions \eqref{dom} and \eqref{ran}
$\dom\gamma=\dom\delta=\Q$ and $\ran\gamma=\ran\delta$. Therefore, by
Proposition~\ref{prop-tlc}, there is a $c\in\Q$ and
$\varepsilon\in\{-1,+1\}$ such that $\delta(t)=\gamma(\varepsilon
t+c)$ for all $t\in\Q$.  By assumption \eqref{sp},
$\gamma(0)=\delta(0)$. Therefore, $c=0$. Since $\gamma$ and $\delta$
are timelike curves $\gamma_1$ and $\delta_1$ are either strictly
increasing or strictly decreasing functions. By assumption \eqref{pd},
$\gamma(1)_1>\gamma(0)_1$ and $\delta(1)_1>\delta(0)_1$. Thus both
$\gamma_1$ and $\delta_1$ are strictly increasing. Consequently,
$\gamma'(0)_1>0$ and $\delta'(0)_1>0$. Therefore, $\varepsilon$ cannot
be negative. Hence we have that $\varepsilon=1$. Consequently,
$\gamma=\delta$ as it was stated.
\end{proof}

\begin{theorem}\label{thm-sh}
Let $d\ge 3$. Assume \ax{AccRel_d} and \ax{Ax\exists UnifOb}. There are
definable differentiable functions $S:\Q\rightarrow\Q$ and
$C:\Q\rightarrow\Q$ with the following properties:
\begin{enumerate}
\item\label{s1} $C^2-S^2=1$,
\item\label{s2} $S(0)=0$ and $C(0)=1$,
\item\label{s3} $S(1)>0$, 
\item\label{ssym} $C(-t)=C(t)$ and $S(-t)=S(t)$ for all $t\in\Q$,
\item\label{smu} $(S')^2-(C')^2=1$,
\item\label{sdiff} $C'=S$ and $S'=C$,
\item\label{smon} $S$ is strictly increasing on $\Q$; and $C$ are
  strictly increasing on interval $[0,\infty)$ and strictly decreasing
    on $(-\infty,0]$.
\item\label{sran} $\ran S=\Q$ and $\ran C=[1,\infty)$.
\end{enumerate}
\end{theorem}

\begin{proof}
Let binary relation on observers $H$ be defined as
\begin{multline}
H(m,k)\defiff\IOb(m)\land \Ob(k)\land\dom
\lc_{km}(t)=\Q\land
\lc_{km}(0)=\langle0,1,0\ldots0\rangle\\\land \lc_{km}(1)_1>0\land
\forall \vx \big[ \vx\in \ran \lc_{km} \leftrightarrow
  x_2^2-x_1^2=1\land
  x_3=\ldots=x_d=0\big].
\end{multline}
Let $\gamma$ be defined as the following relation:
\begin{equation}
\gamma(t)=\vx\defiff \forall mk\enskip H(m,k) \land\lc_{km}(t)=\vx.
\end{equation}
By axiom \ax{Ax\exists UnifOb}, there are such observers $k$ and $m$
that $H(m,k)$ holds.  Therefore relation $\gamma$ is not empty.  By
Lemma~\ref{lem-ulc}, $\gamma$ is a function and it equals to
$\lc_{km}$ for any observers $k$ and $m$ for which relation $H(m,k)$
holds.

$\dom\gamma=\Q$ by $\dom \lc_{km}(t)=\Q$.  By Theorem~\ref{thm-wp},
$\gamma$ is a well-parametrized timelike curve.

Let $C=\gamma_2$ and $S=\gamma_1$. Then $C:\Q\rightarrow\Q$ and
$S:\Q\rightarrow\Q$ are definable differentiable functions since they
are coordinate functions of definable differentiable function
$\gamma(t)=\langle S(t),C(t),0\ldots0\rangle$.

Item \eqref{s1} holds since $\ran\gamma=\{\vx: x_2^2-x_1^2=1\land
x_3=\ldots=x_d=0\}$, Item \eqref{s2} holds since
$\gamma(0)=\langle0,1,0\ldots0\rangle$, and Item \eqref{s3} holds
since $\gamma_1(1)>0$ by the definition of $H(m,k)$.  Item~\eqref{smu}
holds since $\gamma'=\langle S',C',0,\ldots,0\rangle$ because $\gamma$
is well-parametrized.

To prove Item \eqref{ssym}, let us consider curve $\delta:t\mapsto
\langle -S(t), C(t),0\ldots,0\rangle$. It is clear that $\dom \delta
=\dom \gamma$. By Item \eqref{s1}, $\ran\delta =\ran \gamma$.  It is
clear that $\delta$ is also a well-parametrized curve by
Item~\eqref{smu}. Therefore, by Proposition~\ref{prop-tlc},
$\delta(t)=\gamma(\varepsilon t +c)$ for all $t\in\Q$. By
Item~\eqref{s1}, $\delta(0)=\gamma(0)$. Thus $c=0$. By Chain Rule,
$\delta'(t)=\varepsilon\gamma'( \varepsilon t)$. Since both $\delta$
and $\gamma$ are well-parametrized curves,
$\delta(0)=\gamma(0)=\langle0,1,0,\ldots,0\rangle$, and the tangent
line of Hyperbola $\{\vx:x_1^2-x_1^2=1,x_3=\ldots=x_d\}$ is vertical,
we have that $\gamma'_1(0)=0$ and $\gamma'_2(0)=0$. Thus
$\delta(0)=\langle-1,0,\ldots,0\rangle$. Hence $\varepsilon=-1$. Thus
$\delta(t)=\gamma(-t)$. Consequently, $-S(t)=S(-t)$ and $C(t)=C(-t)$.

By Lemma~\ref{lem-mon}, $S$ and $C$ are monotonous on interval $[0,s]$
for all $0<s\in\Q$. Hence they are also monotonous on interval
$[0,\infty)$. By Item \eqref{smu}, $(\gamma'_1)^2\ge1>0$. Therefore,
  by \ax{CONT}-Darboux Theorem, see \cite[\S 10.3]{SzPhd}, $S'(t)>0$
  for all $t\in\Q$ or $S'(t)<0$ for all $t\in\Q$.
  $\gamma(1)_1>\gamma(0)_1$ by Items \eqref{s2} and
  \eqref{s3}. Therefore, $\gamma_t$ is increasing. Thus
  $\gamma'_t>0$. So $S$ is strictly increasing on $\Q$.  Item
  \eqref{s1}, $C$ is strictly increasing on $[0,\infty)$ and strictly
    decreasing on $(-\infty,0]$ since $S$ strictly increasing on $\Q$.

Now let us prove Item \eqref{sdiff}. We have $S'^2-C'^2=1$ by
Item~\eqref{sdiff}. By Chain Rule, if we
differentiate both sides of this equation, we get that
$2SS'-2CC'=0$. Hence $C'C=S'S$.  Multiplying $S'^2-C'^2=1$ by $S^2$,
we get $S'^2S^2-C'^2S^2=S^2$. From this we get $C'^2(C^2-S^2)=S^2$ by
$C'C=S'S$.  Therefore, $C'^2=S^2$ since $C^2-S^2=1$. Consequently,
$C'(t)=\pm S(t)$ and $S'(t)=\pm C(t)$ for all $t\in\Q$. By Items
\eqref{s1} and \eqref{smon}, $S'(t)>0$ and
$C(t)>0$. Therefore, $S'=C$. If $t>0$, a similar argument show that
$C'(t)=S(t)$. By \eqref{ssym}, $-C'(t)=C'(-t)$ and
$S(-t)=-S(t)$. Therefore, $C'(t)=S(t)$ also if $t\in(-\infty, 0]$. By
  Item \eqref{s1} and \eqref{smu}, $C'(0)=0$ and $S'(0)=1$. Hence
  $C'(t)=S(t)$ for all $t\in\Q$.

By \ax{CONT}-Boltzano Theorem\cite[\S 10.2]{SzPhd}, $\ran S$ and $\ran
C$ are intervals. Therefore Item \eqref{sran} holds by Item \ref{s1}.
\end{proof}

\begin{theorem}\label{thm-exp} 
Let $d\ge 3$. Assume \ax{AccRel_d} and \ax{Ax\exists UnifOb}. There is a
definable differentiable function $E:\Q\rightarrow\Q$ with the
following properties:
\begin{enumerate}
\item\label{e2} $E(0)=1$, 
\item\label{e3} $E(1)>0$,
\item\label{esym}  $E(-t)E(t)=1$,
\item\label{ediff}  $E'=E$,
\item\label{eran} $\ran E=(0,\infty)$, and
\item\label{emon} $E$ is strictly increasing.
\end{enumerate}
\end{theorem}
\noindent
Let the {\bf restriction} of function $f$ to set $H$ be defined as 
\begin{equation}
f\big|_H\de\{ \langle x,y\rangle : x\in\dom f\cap H \text{ and } y=f(x)\}
\end{equation}

\begin{proof}
Let $S:\Q\rightarrow\Q$ and $C:\Q\rightarrow\Q$ be the definable
differentiable functions which exist by Theorem~\ref{thm-sh}.  Let
$E:=C+S$. Then $E$ is a definable differentiable function since $C$ and
$S$ are so.  Items \eqref{e2} and \eqref{e3} follow directly from
Items \eqref{s2} and \eqref{s3} of Theorem~\ref{thm-sh}.
Item \eqref{esym} follows from Items \eqref{s1} and \eqref{ssym} of
Theorem~\ref{thm-sh} because
$E(-t)E(t)=\big(C(-t)+S(-t)\big)\big(C(t)+S(t)\big)=C^2(t)-S^2(t)=1$.
Item \eqref{eran} follows from Item \eqref{sdiff} of
Theorem~\ref{thm-sh}.
Item \eqref{ediff} follows because of the following.  $\ran
E|_{[0,\infty)}=[1,\infty)$ because $E(0)=1$, $S$ and $C$ are strictly
    increasing on $[0,\infty)$, and $\ran C=[1,\infty)$ by Item
        \eqref{sran} of Theorem~\ref{thm-sh}.  Hence $\ran
        E|_{(-\infty,0]}=(0,1]$ by Item \eqref{esym}. Thus $\ran
    E=(0,\infty)$.
Item \eqref{emon} follows from Item \eqref{sdiff} of
Theorem~\ref{thm-sh} since $E$ is strictly increasing on $[0,\infty)$
  by Item \eqref{smon} of Theorem~\ref{thm-sh} and $E$ is also
  strictly increasing on $(-\infty,0]$ since $E(-t)E(t)=1$ by Item
\eqref{esym}.
\end{proof}

The following first-order logic formula defines that {\bf limit of
  function} $f$ is $A$ at $x_0$ over every ordered field:
\begin{multline}
\limit(f,x_0,A)\defiff \forall \varepsilon>0\; \exists \delta >0\;\forall x
\\ \enskip 0<|x-x_0|<\delta\land  x\in \dom f\rightarrow 
|f(x)-A|<\varepsilon.
\end{multline}

In the case when there is one and only one $A$ such that
$\limit(f,x_0,A)$ holds, we write $\lim_{x\rightarrow x_0}f(x)=A$.  By
using the technique of \cite[\S 10]{SzPhd}, it can be proved that
there is at most one $A$ such that $\limit(f,x_0,A)$ holds if $x_0$ is
a accumulation point of $\dom f$ (i.e., if for all $\varepsilon>0$,
there is $x\in\dom f$ such that $|x-x_0|<\varepsilon$).

Let the exponential function of $\R$ be denoted by $\exp$.

\begin{proposition}\label{prop-exp}
Let $\langle\Q,+,\cdot,\le\rangle$ be a subfield of $\R$.\footnote{By
  Pickert--Hion Theorem, these fields are exactly the fields of
  Acrhiedean ordered fields, see, e.g., \cite[\S VIII]{fuchs},
  \cite[C.44.2]{CHA}.} Let $f:\Q\rightarrow\Q$ be a differentiable
function such that $f'=f$ and $f(0)=1$. Then $f=\exp\big|_{\Q}$, i.e.,
$f$ is the restriction of the real exponential function to $\Q$.
\end{proposition}

\begin{proof}
We have that $f$ is continuous since $f$ is differentiable, see, e.g.,
\cite[Cor.10.3.5]{SzPhd}.  Let $f_*(x)\de \lim_{t\rightarrow
  x}f(t)$. Function $f_*$ is well defined since $f$ is continuous and
$\Q$ is dense in $\R$ (as $\rac\subseteq\Q$).  Since $f$ is
continuous, we also have that $f_*$ is an extension of $f$, i.e.,
$f_*(x)= f(x)$ if $x\in\Q$. We are going to show that $f_*=\exp$.
First we show that $f_*'(x)=f_*(x)$. We start by showing that, for all
$x,y\in\R$ and $\varepsilon_0>0$, there are $x^*,y^*\in\Q$ such that
$|x-x^*|<\varepsilon_0$, $|y-y^*|<\varepsilon_0$, and
\begin{equation}
\left|\frac{f(x^*)-f(y^*)}{x^*-y^*}-\frac{f_*(x)-
  f_*(y)}{x-y}\right|<\varepsilon_0.
\end{equation}
By the triangle inequality, 
\begin{multline}
\left|\frac{f_*(x)-f_*(y)}{x-y}-\frac{f(x^*) -f(y^*)}{x^*-y^*} \right|
 \le\\\left|\frac{f_*(x)- f_*(y)}{x-y}- \frac{f(x^*)-
  f(y^*)}{x-y}+ \frac{f(x^*) -f(y^*)}{x-y} -
\frac{x-y}{x^*-y^*} \cdot\frac{f(x^*)-f(y^*)}{x-y}\right| \\\le
\left|\frac{f_*(x)- f(x^*)}{x-y}\right| +
+\left|\frac{f_*(y)-f(y^*)}{x-y}\right| +\left|1-\frac{x-y}{x^*-y^*}
\right|\cdot\left|\frac{f(x^*)-f(y^*)}{x-y}\right|.
\end{multline}
By the definition of $f_*$, there is a $\delta$ such that
\begin{equation}
|f_*(x)-f(x^*)|<\frac{\varepsilon|x-y|}{3} \enskip \text{ and
} \enskip |f_*(y)- f(y^*)|<\frac{\varepsilon|x-y|}{3}
\end{equation}
 if $|x-y|<\delta$.  From this, by the triangle inequality, we have
 that
\begin{multline}
|f(x^*)-f(y^*)|\le |f(x^*)-f_*(x)| + |f_*(x)-f_*(y)| + | f_*(y)-
f(y^*)|\\< |f_*(x)-f_*(y)|+ \frac{2\varepsilon_0|y-x|}{3}.
\end{multline}

Since $\left|1-\frac{x-y}{x^*-y^*}\right|$ can be arbitrarily small if
$|x-x^*|$ and $|y-y^*|$ are small enough, we can choose $x^*$ and
$y^*$ such that
\begin{equation}
\left|1-\frac{x-y}{x^*-y^*}\right|\cdot\left|\frac{f(x^*)-f(y^*)}{x-y}\right|<\frac{\varepsilon_0}{3}.
\end{equation}
Therefore, there are $x^*$ and $y^*$ arbitrarily close to $x$ and $y$
such that
\begin{equation}\label{eq-0}
 \left|\frac{f_*(x)-f_*(y)}{x-y}-\frac{f(x^*)-f(y^*)}{x^*-y^*}\right|<\varepsilon_0.
\end{equation}
To prove that $f'_*=f_*$, We have to show that, for all
$\varepsilon>0$, there is a $\delta>0$ such that
\begin{equation}\label{eq-df}
 \left|f_*(x)-\frac{f_*(x)-f_*(y)}{x-y}\right|<\varepsilon
\end{equation}
if $|x-y|<\delta$. 
By the triangle inequality,
\begin{multline}\label{eq-tri}
 \left|f_*(x)-\frac{f_*(x)-f_*(y)}{x-y}\right|\le |f_*(x)-f(x^*)| \\ +\left| f(x^*)-\frac{f(x^*)-f(y^*)}{x^*-y^*}\right|+\left|\frac{f(x^*)-f(y^*)}{x^*-y^*}  -\frac{f_*(x)-f_*(y)}{x-y}\right|.
\end{multline}
By the definition of $f_*$,  
\begin{equation}\label{eq-a}
|f_*(x)-f(x^*)|<\frac{\varepsilon}{3}
\end{equation}
if $|x-x^*|$ small enough.
By \eqref{eq-0}, we have that there, are $x^*$ and $y^*$ arbitrarily
close to $x$ and $y$ such that,
\begin{equation}\label{eq-b}
  \left|\frac{f(x^*)-f(y^*)}{x^*-y^*}  -\frac{f_*(x)-f_*(y)}{x-y}\right|<\frac{\varepsilon}{3}.
\end{equation}
Since $f'=f$, we have that 
\begin{equation}\label{eq-diff}
\left| f(x^*)-\frac{f(x^*)-f(y^*)}{x^*-y^*}\right|<\frac{\varepsilon}{3}
\end{equation}
if $|x^*-y^*|$ is small enough. So, if $|x-y|$ is small enough and we
can choose $x^*$ and $y^*$ close enough to $x$ and $y$ we have that
Eq.~\eqref{eq-diff} holds. Consequently, if $|x-y|$ is small
enough, then Eq.~\eqref{eq-df} holds, i.e., $f_*$ is
differentiable and $f_*'=f_*$.  Therefore, there is a $c\in\R$ such
that $f_*(x)=c \exp(x)$ for all $x\in\R$. We have that $c=1$ since
$c=c\exp(0)=f_*(0)=f(0)=1$. Therefore, $f$ is the restriction of
function $\exp$ to $\Q$; and this is what we wanted to prove.
\end{proof}

\begin{proof}[Proof of Theorem\ref{thm-noalg}]\label{proof-noalg}
By Theorem~\ref{thm-exp}, a differentiable function $E$ is definable
in the models of \ax{AccRel_d} + \ax{Ax\exists UnifOb} such that $E'=E$
and $E(0)=1$.  By Proposition~\ref{prop-exp}, $E$ has to be the
restriction of $\exp$ to the real algebraic numbers. However, this is
impossible since then $E(1)$ is the Euler-number $e$ which is not an
algebraic number.
\end{proof}

\section{Acknowledgments}
This research is supported by the Hungarian Scientific Research Fund
for basic research grants No.~T81188 and No.~PD84093.

\bibliography{LogRelBib} \bibliographystyle{plain}

\end{document}